\documentclass[11pt,a4paper]{article}
\usepackage[utf8]{inputenc}
\usepackage[english]{babel}
\usepackage{amsmath}
\usepackage{amsthm}
\usepackage{amsfonts}
\usepackage{amssymb}
\usepackage{graphicx}
\usepackage{booktabs}
\usepackage[hidelinks]{hyperref}
\usepackage{subcaption}
\usepackage{multirow}
\usepackage{tikz} 
\usepackage{verbatim}
\usepackage[ruled,vlined]{algorithm2e}
\usepackage{enumitem}
\usepackage{float}
\usepackage{algorithmic}

\usetikzlibrary{arrows,shapes,positioning}
\tikzstyle{block} = [draw, rectangle, minimum width = 0.75cm, minimum height = 0.75cm]
\tikzstyle{sum} = [draw, circle, minimum size=.5cm, node distance=1.75cm]
\tikzstyle{input} = [coordinate]
\tikzstyle{output} = [coordinate]

\numberwithin{equation}{section}

\newtheorem{theorem}{Theorem}[section]

\newtheorem{remark}{Remark}[section]

\makeatletter
\newcommand{\startat}[1]{\setcounter{\@enumctr}{#1}%
	\addtocounter{\@enumctr}{-1}}
\makeatother

\newcommand{\normiii}[1]{{\vert\kern-0.25ex\vert\kern-0.25ex\vert #1 
		\vert\kern-0.25ex\vert\kern-0.25ex\vert}}

\makeatletter
\newcommand{\mylabel}[2]{#2\def\@currentlabel{#2}\label{#1}}

\title{\bf Analysis and Optimal Design of Equilibria in the Vlasov-Poisson System}
\author{A. Borz{\`i}
	\thanks{Institut f\"ur Mathematik, Universit\"at W\"urzburg, Emil-Fischer-Strasse 30,
		97074 W\"urzburg, Germany. e-mail: alfio.borzi@mathematik.uni-wuerzburg.de}
	\and 
	G. Infante
	\thanks{Dipartimento di Matematica e Informatica, Universit\`a della Calabria, 87036 Arcavacata di Rende, Cosenza, Italy
		e-mail: gennaro.infante@unical.it}
	\and
	G. Mascali
	\thanks{Dipartimento di Matematica e Informatica, Universit\`a della Calabria, 87036 Arcavacata di Rende, Cosenza, Italy, and INFN gruppo collegato Cosenza, Italy
		e-mail: giovanni.mascali@unical.it}
	\and
R. Scida
\thanks{Dipartimento di Matematica e Informatica, Universit\`a della Calabria, 87036 Arcavacata di Rende, Cosenza, Italy
	e-mail: riccardo.scida@unical.it}
}
\date{\today}

\begin{document}
	
	\maketitle

\begin{abstract}
The optimal design of equilibrium particle distributions generated by
external electric fields is investigated in the framework of stationary
Maxwellian equilibria of the Vlasov-Poisson system. The resulting
normalized Poisson-Boltzmann equation is analyzed, and existence and
uniqueness of the self-consistent equilibrium potential are established
by two complementary approaches. A Schauder fixed-point argument based
on uniform estimates for the normalized nonlinear source and a
variational approach based on a strictly convex free-energy functional
are developed. Building upon these analytical results, a unified
optimal design framework is formulated for equilibrium densities.
Valley-potential, quadratic tracking, and Kullback-Leibler design
criteria are treated within the same optimization framework.
First-order optimality conditions are derived, and a projected
nonlinear conjugate-gradient algorithm is proposed for the numerical
solution of the resulting optimization problems.
\end{abstract}

\section{Introduction}

Many physical systems evolve over long time intervals toward stationary
or metastable equilibrium configurations. In plasma physics,
charged-particle beams, semiconductor devices, and trapped particle
systems, the equilibrium density determines the spatial distribution of
the particles and therefore directly influences the operating
characteristics of the device. An important question is whether such
equilibrium configurations can be designed through suitable external
fields in order to achieve prescribed objectives, such as particle
concentration in selected regions, confinement, or approximation of a
desired density profile.

This question naturally leads to optimal design problems in which the
control variable is an externally applied electric potential, whereas
the state is given by the corresponding self-consistent equilibrium
density. Such problems combine nonlinear kinetic equilibrium theory,
elliptic partial differential equations, and optimization. 
Our mathematical model is obtained by combining
the stationary Maxwellian equilibrium of the Vlasov-Poisson system with
the associated self-consistent Poisson equation, leading to a normalized
Poisson-Boltzmann equation that serves as the state equation in the 
formulation of our optimal design problems. 

Nonlinear Poisson-Boltzmann equations arising from kinetic and
statistical-mechanical models have been investigated in several
settings, including plasma physics, stellar dynamics, and
semiconductor theory
\cite{Bavaud1991,Carrillo1998,DesvillettesDolbeault1991,GognyLions1989}. 
In the present work, we consider a normalized Poisson-Boltzmann equation
arising from Maxwellian equilibria in the presence of an externally
controlled electric field. We remark that the normalization introduces a nonlocal
dependence through the total-mass constraint and naturally fits within
an optimal design framework.

The evolution of collisionless plasmas is commonly described by the
Vlasov-Poisson system, whose mathematical theory has been extensively
developed over the last decades; see, for example,
\cite{Arsenev1975,DiPernaLions1988,HorstHunzeNeunzert1984,
LionsPerthame1991,Pfaffelmoser1992}.
Stationary solutions play a central role in plasma theory since they
describe equilibrium configurations around which the dynamics evolves.
Under suitable assumptions on the temperature and on the external
electric field, Maxwellian equilibria give rise to nonlinear
Poisson-Boltzmann equations relating the self-consistent electric
potential to the equilibrium density. The equilibrium potential determines the electric field 
generated by the particle distribution, while the corresponding equilibrium density is
described by a normalized Maxwell-Boltzmann distribution. This
normalization enforces conservation of the total particle mass and
introduces a nonlocal coupling into the resulting elliptic equation.

Equilibrium equations of Poisson-Boltzmann type have been investigated
in several contexts, including plasma physics, stellar dynamics,
semiconductor models, and statistical mechanics.
Variational formulations and existence results were developed by
Gogny and Lions \cite{GognyLions1989}, while the corresponding
equilibrium potential arising in the asymptotic analysis of the
Vlasov-Poisson-Boltzmann system was further investigated by
Desvillettes and Dolbeault \cite{DesvillettesDolbeault1991}.
Related nonlinear elliptic equations with decreasing nonlinearities have
also been analyzed by Carrillo \cite{Carrillo1998}, whereas a broader
statistical-mechanical perspective is provided by
Bavaud \cite{Bavaud1991}.

The optimal design of equilibrium particle distributions has recently
been considered in \cite{BorziInfanteMascali2025}, where external
electric fields were optimized to generate equilibrium densities
concentrated near prescribed potential wells. That formulation relied on a 
particular valley-potential criterion and
the corresponding optimality system. Many practical situations,
however, require different notions of optimality, as the
approximation of prescribed equilibrium densities or the minimization
of information-theoretic measures of discrepancy. 
These considerations motivate our development of a more general
optimization framework capable of accommodating different design
criteria while preserving the same underlying equilibrium model.

Our work substantially extends this framework by replacing the
specific valley-potential formulation with a unified optimal design
approach capable of accommodating several classes of design criteria,
including valley-potential functionals, quadratic tracking terms, and
Kullback-Leibler-type distances from prescribed target densities.

We discuss two independent existence and uniqueness theories.
The first one combines Schauder's fixed-point theorem with uniform
elliptic estimates and explicit bounds on the normalization factor.
The second one is variational and relies on the strict convexity of an
appropriate free-energy functional, thereby providing an alternative and
conceptually different proof.

Building upon these analytical results, we formulate a unified optimal design
framework. 
The corresponding first-order optimality conditions are derived through
an adjoint approach, leading to a variational inequality for constrained
controls and to an unconstrained optimality equation.
A projected nonlinear conjugate-gradient method is then outlined for the
numerical solution of the resulting optimization problem, and 
implemented to validate our optimal design strategy, 

The remainder of the paper is organized as follows. Section~2 derives the
normalized Poisson-Boltzmann equation from stationary Maxwellian equilibria of the
Vlasov-Poisson system. Section~3 proves existence by means of the Schauder
fixed-point theorem, while Section~4 establishes uniqueness through a monotonicity
argument. Section~5 provides an alternative variational proof of existence and
uniqueness. Section~6 formulates the general optimal design problem for equilibrium
densities, and Section~7 derives the corresponding first-order optimality conditions.
Section~8 outlines the numerical solution strategy, whereas Section~9 presents the
results of the numerical experiments.

\section{From the Vlasov-Poisson system to the normalized Poisson-Boltzmann equation}
\label{sec:derivation}

In this section, we recall how the normalized Poisson-Boltzmann equation arises
from stationary equilibrium solutions of the Vlasov-Poisson system in the presence
of an external electric field. 
We consider a plasma of particles confined in a bounded domain 
\(\Omega\subset\mathbb R^n\) and described by the 
following distribution function
\[
f=f(x,v,t),
\]
where \(x\in\Omega\) is the spatial variable, \(v\in\mathbb R^n\) is the particle
velocity, and \(t\geq 0\) is time. We assume that \(\Omega\) has sufficiently smooth
boundary.

In the zero magnetic field regime, and after a suitable nondimensionalization, the
following Vlasov equation is obtained
\[
\partial_t f+v\cdot\nabla_x f+\bigl(E(x,t)+E_0(x)\bigr)\cdot\nabla_v f=0,
\]
where \(\nabla_x\) and \(\nabla_v\) denote the gradients with respect to the
spatial and velocity variables, respectively.
The field \(E(x,t)\) is the self-consistent electric field generated by the particles,
whereas \(E_0(x)\) is a prescribed stationary external electric field. We assume that
both fields are conservative, namely
\[
E(x,t)=-\nabla_x U(x,t),
\qquad
E_0(x)=-\nabla_x u(x),
\]
where \(U\) is the self-consistent potential and \(u\) is the external potential.

For particles of the same positive charge, the self-consistent field is coupled to
the spatial density
\[
\rho(x,t)=\int_{\mathbb R^n} f(x,v,t)\,dv
\]
through the Poisson equation
\[
\operatorname{div}E(x,t)=\rho(x,t).
\]
Since \(E=-\nabla_x U\), this gives
\[
-\Delta U(x,t)=\rho(x,t),
\]
where \(\Delta\) denotes the Laplacian with respect to the spatial variable \(x\).

We now look for stationary Maxwellian solutions with zero mean velocity and
constant temperature \(T>0\). Thus, we assume that
\[
f_s(x,v)
=
\frac{1}{(2\pi T)^{n/2}}\rho(x)
\exp\left(-\frac{|v|^2}{2T}\right).
\]
Substituting this expression into the stationary Vlasov equation gives
\[
v\cdot\nabla_x f_s
+
(E+E_0)\cdot\nabla_v f_s=0.
\]
Since
\[
\nabla_v f_s=-\frac{v}{T}f_s,
\qquad
\nabla_x f_s=f_s\nabla_x\log\rho,
\]
we obtain
\[
f_s\,v\cdot
\left(
\nabla_x\log\rho-\frac{E+E_0}{T}
\right)=0.
\]
Therefore, we have
\[
\nabla_x\log\rho
=
\frac{E+E_0}{T}
=
-\frac{1}{T}\bigl(\nabla_x U+\nabla_x u\bigr).
\]
Equivalently, we write
\[
\nabla_x
\left(
\log\rho+\frac{U+u}{T}
\right)=0.
\]
Hence there exists a constant \(C\) such that
\[
\rho(x)=C\exp\left(-\frac{U(x)+u(x)}{T}\right).
\]

We impose the normalization
\[
\int_\Omega \rho(x)\,dx=1.
\]
This gives
\[
1
=
C\int_\Omega
\exp\left(-\frac{U(y)+u(y)}{T}\right)\,dy,
\]
and therefore
\[
C
=
\left(
\int_\Omega
\exp\left(-\frac{U(y)+u(y)}{T}\right)\,dy
\right)^{-1}.
\]
Thus the spatial density is given by
\[
\rho(x)
=
\frac{
	\exp\left(-\frac{U(x)+u(x)}{T}\right)
}{
	\int_\Omega
	\exp\left(-\frac{U(y)+u(y)}{T}\right)\,dy
}.
\]
Combining this identity with the Poisson equation \(-\Delta U=\rho\), and imposing
homogeneous Dirichlet boundary conditions, we obtain the nonlinear elliptic problem
for the stationary self-consistent potential:
\begin{equation}
	\begin{cases}
		-\Delta U(x)
		=
		\dfrac{
			\rho_0(x)\exp\left(-\dfrac{U(x)}{T}\right)
		}{
			\displaystyle\int_\Omega
			\rho_0(y)\exp\left(-\dfrac{U(y)}{T}\right)\,dy
		},
		& x\in\Omega,\\[2ex]
		U=0,
		& x\in\partial\Omega,
	\end{cases}
	\label{eq:Poisson-Boltzmann}
	\tag{P}
\end{equation}
where we have introduced the weight
\[
\rho_0(x):=\exp\left(-\frac{u(x)}{T}\right).
\]
This is the form that will be used in the following sections to prove existence and
uniqueness of solutions.
\begin{remark}
	If the total mass is not normalized to one, but instead
	\[
	\int_\Omega \rho(x)\,dx=M>0,
	\]
	then the same derivation gives
	\[
	-\Delta U(x)
	=
	M
	\frac{
		\exp\left(-\dfrac{U(x)+u(x)}{T}\right)
	}{
		\displaystyle\int_\Omega
		\exp\left(-\dfrac{U(y)+u(y)}{T}\right)\,dy
	}.
	\]
	In this work we restrict ourselves to the normalized case \(M=1\).
\end{remark}
\section{Existence via the Schauder fixed-point theorem}
\label{sec:existence-iterative}

In this section, we prove the existence of solutions to the normalized
Poisson-Boltzmann problem \eqref{eq:Poisson-Boltzmann}. 
Throughout this section we assume that
\[
\Omega\subset\mathbb R^n,\qquad n\leq 3,
\]
is a bounded domain of class \(C^{1,1}\), that \(T>0\), and that the external
potential \(u\) satisfies
\[
u\in L^1(\Omega),
\qquad
\rho_0=e^{-u/T}\in L^2(\Omega).
\]
The first assumption is used in the Jensen-type estimate which provides a lower
bound for the normalization factor, while the second one ensures that the
right-hand side of the equation belongs to \(L^2(\Omega)\).

We begin by introducing the nonlinear map which will be used in the fixed-point
argument. Given a function \(W\), we define \(U=\mathcal T(W)\), where \(U\) is the
solution of the linear Dirichlet problem
\[
\begin{cases}
	-\Delta U=
	\dfrac{
		\rho_0(x)\exp\left(-\dfrac{W(x)}{T}\right)
	}{
		\displaystyle\int_\Omega
		\rho_0(y)\exp\left(-\dfrac{W(y)}{T}\right)\,dy
	},
	& x\in\Omega,\\[2ex]
	U=0,
	& x\in\partial\Omega.
\end{cases}
\]
A fixed point of \(\mathcal T\), namely a function \(U\) such that
\[
\mathcal T(U)=U,
\]
is precisely a solution of the normalized Poisson-Boltzmann problem
\eqref{eq:Poisson-Boltzmann}.

We now prove that \(\mathcal T\) has a fixed point.

\begin{theorem}
	Let \(\Omega\subset\mathbb R^n\), \(n\leq 3\), be a bounded \(C^{1,1}\) domain.
	Let \(T>0\), and assume that
	\[
	u\in L^1(\Omega),
	\qquad
	e^{-u/T}\in L^2(\Omega).
	\]
	Then problem \((P)\) admits at least one solution
	\[
	U\in H^2(\Omega)\cap H_0^1(\Omega).
	\]
	Moreover,
	\[
	U\geq 0
	\qquad\text{a.e. in }\Omega.
	\]
\end{theorem}

\begin{proof}
Let \(\zeta\) be the torsion function of the domain \(\Omega\), that is, the solution
of the classical torsion problem
\[
\begin{cases}
	-\Delta \zeta=1,
	& x\in\Omega,\\
	\zeta=0,
	& x\in\partial\Omega.
\end{cases}
\]
	By elliptic regularity and the maximum principle, \(\zeta\in L^\infty(\Omega)\) and
	\(\zeta\geq 0\). We set
	\[
	\zeta_\Omega:=\|\zeta\|_{L^\infty(\Omega)}.
	\]
	
	Let \(W\in H^2(\Omega)\cap H_0^1(\Omega)\) be a nonnegative function such that
	\[
	\int_\Omega W(x)\,dx\leq \zeta_\Omega.
	\]
	Define
	\[
	A(W):=
	\int_\Omega \rho_0(y)\exp\left(-\frac{W(y)}{T}\right)\,dy
	=
	\int_\Omega
	\exp\left(-\frac{W(y)+u(y)}{T}\right)\,dy.
	\]
	Since \(W\geq 0\) and \(\rho_0\in L^2(\Omega)\subset L^1(\Omega)\), we have
	\[
	0<A(W)\leq \int_\Omega \rho_0(y)\,dy<+\infty.
	\]
	
	We next prove a uniform lower bound for \(A(W)\). By Jensen's inequality,
	\[
	A(W)
	=
	\int_\Omega
	\exp\left(-\frac{W(y)+u(y)}{T}\right)\,dy
	\geq
	|\Omega|
	\exp\left[
	-\frac{1}{T|\Omega|}
	\int_\Omega (W(y)+u(y))\,dy
	\right].
	\]
	Using
	\[
	\int_\Omega W(y)\,dy\leq \zeta_\Omega,
	\]
	we obtain
	\[
	A(W)\geq
	|\Omega|
	\exp\left[
	-\frac{1}{T}
	\left(
	\frac{\zeta_\Omega}{|\Omega|}
	+
	\frac{1}{|\Omega|}\int_\Omega u(y)\,dy
	\right)
	\right]
	=:A_* >0.
	\]
	Therefore, if
	\[
	F_W(x):=
	\dfrac{
		\rho_0(x)\exp\left(-\dfrac{W(x)}{T}\right)
	}{
		A(W)
	},
	\]
	then
	\[
	0\leq F_W(x)\leq \frac{\rho_0(x)}{A_*}
	\qquad\text{a.e. in }\Omega.
	\]
	Consequently, we obtain the estimate
	\[
	\|F_W\|_{L^2(\Omega)}
	\leq
	\frac{1}{A_*}\|\rho_0\|_{L^2(\Omega)}.
	\]
	Moreover, by definition of \(F_W\), it follows that 
	\[
	\int_\Omega F_W(x)\,dx=1.
	\]
	
	Let now \(U=\mathcal T(W)\), namely
	\[
	\begin{cases}
		-\Delta U=F_W,
		& x\in\Omega,\\
		U=0,
		& x\in\partial\Omega.
	\end{cases}
	\]
	Since \(F_W\geq 0\), the weak maximum principle gives
	\[
	U\geq 0
	\qquad\text{a.e. in }\Omega.
	\]
Testing the equation for \(U\) with \(\zeta\), and the equation for \(\zeta\) with
\(U\), we obtain
\[
\int_\Omega U(x)\,dx
=
\int_\Omega \zeta(x)F_W(x)\,dx.
\]
Therefore, we have
\[
\int_\Omega U(x)\,dx
\leq
\|\zeta\|_{L^\infty(\Omega)}
\int_\Omega F_W(x)\,dx
=
\zeta_\Omega.
\]	
Thus the map \(\mathcal T\) preserves the uniform \(L^1\) bound.
	
By elliptic \(H^2\)-regularity for the Dirichlet problem on \(C^{1,1}\) domains,
	there exists a constant \(C_{\rm reg}>0\), depending only on \(\Omega\), such that
	\[
	\|U\|_{H^2(\Omega)}
	\leq
	C_{\rm reg}\|F_W\|_{L^2(\Omega)}.
	\]
	Hence
	\[
	\|\mathcal T(W)\|_{H^2(\Omega)}
	\leq
	\frac{C_{\rm reg}}{A_*}\|\rho_0\|_{L^2(\Omega)}
	=:C_*.
	\]
	Since \(n\leq 3\), the Sobolev embedding theorem gives 
	$H^2(\Omega)\hookrightarrow C(\overline\Omega)$. Therefore
	\[
	\|\mathcal T(W)\|_{C(\overline\Omega)}
	\leq C_{\rm emb} C_*,
	\]
	where \(C_{\rm emb}>0\) is the embedding constant.
	
	We set $M:=C_{\rm emb}C_*$, and introduce the closed convex set
	\[
	\mathcal K:=
	\left\{
	W\in C(\overline\Omega):
	W\geq 0,\ 
	W=0 \text{ on }\partial\Omega,\ 
	\int_\Omega W(x)\,dx\leq \zeta_\Omega,\ 
	\|W\|_{C(\overline\Omega)}\leq M
	\right\}.
	\]
	The previous estimates show that $\mathcal T(\mathcal K)\subseteq \mathcal K$.

	We now remark that \(\mathcal T\) is continuous from \(\mathcal K\) into
	\(C(\overline\Omega)\). Indeed, if \(W_j\to W\) in \(C(\overline\Omega)\), then
	\[
	\exp\left(-\frac{W_j}{T}\right)
	\to
	\exp\left(-\frac{W}{T}\right)
	\]
	uniformly in \(\overline\Omega\). Since \(\rho_0\in L^2(\Omega)\), this implies
	\[
	F_{W_j}\to F_W
	\qquad\text{in }L^2(\Omega).
	\]
	By elliptic regularity,
	\[
	\mathcal T(W_j)\to \mathcal T(W)
	\qquad\text{in }H^2(\Omega),
	\]
	and therefore also in \(C(\overline\Omega)\).
	
	Moreover, \(\mathcal T(\mathcal K)\) is bounded in \(H^2(\Omega)\). Since
	\(H^2(\Omega)\) is compactly embedded into \(C(\overline\Omega)\) for \(n\leq 3\),
	the map \(\mathcal T:\mathcal K\to\mathcal K\) is compact.
	
	By Schauder's fixed point theorem, there exists \(U\in\mathcal K\) such that
	\[
	\mathcal T(U)=U.
	\]
	Hence, \(U\) satisfies the equation
	\[
	-\Delta U
	=
	\dfrac{
		\rho_0(x)\exp\left(-\dfrac{U(x)}{T}\right)
	}{
		\displaystyle\int_\Omega
		\rho_0(y)\exp\left(-\dfrac{U(y)}{T}\right)\,dy
	}
	\]
	with homogeneous Dirichlet boundary conditions. The \(H^2\)-regularity follows
	from the estimate above, and the nonnegativity follows from the maximum principle.
	This concludes the proof.
\end{proof}

\section{Uniqueness of the solution}
\label{sec:uniqueness}

The uniqueness of solutions to the normalized Poisson-Boltzmann equation can be
proved by a comparison argument, following the strategy used for the equilibrium
equation in the optimal design framework. More precisely, one compares two
solutions through their normalizing factors and then applies a maximum/minimum
principle argument.

In this section, we present an alternative proof, based on the monotonicity of the normalized
Gibbs density. This argument is short and will also be useful in connection with the
variational formulation discussed in the next section.

Let \(U_1,U_2\in H^2(\Omega)\cap H_0^1(\Omega)\) be two solutions of
problem \eqref{eq:Poisson-Boltzmann}.
Set
\[
\rho_i(x)
:=
\frac{
	\exp\left(-\dfrac{U_i(x)+u(x)}{T}\right)
}{
	\displaystyle\int_\Omega
	\exp\left(-\dfrac{U_i(y)+u(y)}{T}\right)\,dy
},\quad i=1,2.
\]
Then, we have
\[
-\Delta U_i=\rho_i,
\qquad i=1,2.
\]
Next, we define $w:=U_1-U_2$, and subtracting the two equations, we obtain
\[
-\Delta w=\rho_1-\rho_2,
\qquad w=0 \quad \text{on } \partial\Omega.
\]
Testing this equation with \(w\), we get
\[
\int_\Omega |\nabla w|^2\,dx
=
\int_\Omega (\rho_1-\rho_2)w\,dx.
\]
We now prove that the right-hand side is nonpositive.

For \(s\in[0,1]\), define
\[
U_s:=U_2+s(U_1-U_2)=U_2+sw
\]
and
\[
\rho_s(x)
:=
\frac{
	\exp\left(-\dfrac{U_s(x)+u(x)}{T}\right)
}{
	\displaystyle\int_\Omega
	\exp\left(-\dfrac{U_s(y)+u(y)}{T}\right)\,dy
}.
\]
Next, consider the function
\[
m(s):=\int_\Omega \rho_s(x)w(x)\,dx.
\]
Differentiating with respect to \(s\), we obtain
\[
m'(s)
=
-\frac{1}{T}
\left[
\int_\Omega \rho_s(x)w(x)^2\,dx
-
\left(\int_\Omega \rho_s(x)w(x)\,dx\right)^2
\right].
\]
Since \(\rho_s\geq 0\) and
\[
\int_\Omega \rho_s(x)\,dx=1,
\]
the term in brackets is the variance of \(w\) with respect to the probability density
\(\rho_s\). Hence, it holds
\[
m'(s)\leq 0
\qquad\text{for every }s\in[0,1].
\]
Therefore \(m\) is nonincreasing, and
\[
\int_\Omega (\rho_1-\rho_2)w\,dx
=
m(1)-m(0)
\leq 0.
\]
Consequently, we have 
\[
\int_\Omega |\nabla w|^2\,dx\leq 0.
\]
It follows that
\[
\nabla w=0
\qquad\text{a.e. in }\Omega.
\]
Since \(w\in H_0^1(\Omega)\), we conclude that
\[
w=0
\qquad\text{a.e. in }\Omega.
\]
Thus
\[
U_1=U_2,
\]
and the solution is unique.

\section{A variational proof inspired by Gogny-Lions and Desvillettes-Dolbeault}
\label{sec:variational-proof}

In this section, we give a second proof of existence and uniqueness, based on a
variational formulation. The argument is inspired by the variational approach used
by Gogny and Lions for equilibrium electron densities in plasmas and by
Desvillettes and Dolbeault for the equilibrium potential arising in the long-time
asymptotics of the Vlasov-Poisson-Boltzmann equation; see 
\cite{DesvillettesDolbeault1991,GognyLions1989}. 

We consider again problem \eqref{eq:Poisson-Boltzmann},
and we assume that \(\Omega\subset\mathbb R^n\), \(n\leq 3\), is a bounded
\(C^{1,1}\) domain, that \(T>0\), and that
\[
u\in L^1(\Omega),
\qquad
\rho_0=e^{-u/T}\in L^2(\Omega).
\]

The natural energy associated with \((P)\) is given by
\[
\mathcal E(U)
=
\frac12\int_\Omega |\nabla U|^2\,dx
+
T\log\left(
\int_\Omega \rho_0(x)\exp\left(-\frac{U(x)}{T}\right)\,dx
\right).
\]
We regard \(\mathcal E\) as an extended functional on \(H_0^1(\Omega)\), with
values in \((-\infty,+\infty]\).

\begin{theorem}
	Under the assumptions above, the functional \(\mathcal E\) admits a unique minimizer
	in \(H_0^1(\Omega)\). This minimizer is the unique weak solution of \((P)\).
	Moreover, it holds
	\[
	U\in H^2(\Omega)\cap H_0^1(\Omega),
	\qquad
	U\geq 0
	\quad\text{a.e. in }\Omega.
	\]
\end{theorem}

\begin{proof}
	We divide the proof into several steps.
	
	\medskip
	
	\noindent
	\textit{Step 1: reduction to nonnegative functions.}
	Let \(U\in H_0^1(\Omega)\), and denote by
	\[
	U^+(x):=\max\{U(x),0\}
	\]
	its positive part. Since $|\nabla U^+|\leq |\nabla U|$ a.e. in $\Omega$, we have
	\[
	\frac12\int_\Omega |\nabla U^+|^2\,dx
	\leq
	\frac12\int_\Omega |\nabla U|^2\,dx.
	\]
	Moreover, since \(U^+\geq U\), we have
	\[
	\exp\left(-\frac{U^+}{T}\right)
	\leq
	\exp\left(-\frac{U}{T}\right),
	\]
	and therefore
	\[
	\int_\Omega \rho_0(x)\exp\left(-\frac{U^+(x)}{T}\right)\,dx
	\leq
	\int_\Omega \rho_0(x)\exp\left(-\frac{U(x)}{T}\right)\,dx.
	\]
	It follows that
	\[
	\mathcal E(U^+)\leq \mathcal E(U).
	\]
	Thus, in the minimization of \(\mathcal E\), it is not restrictive to consider
	nonnegative functions.
	
	\medskip
	
	\noindent
	\textit{Step 2: coercivity.}
	Let \(U\in H_0^1(\Omega)\), \(U\geq 0\). We recall that, by Jensen's inequality, 
	we have
	\[
	A(U)
	\geq
	|\Omega|
	\exp\left[
	-\frac{1}{T|\Omega|}
	\int_\Omega (U(x)+u(x))\,dx
	\right].
	\]
	Hence the estimate
	\[
	T\log A(U)
	\geq
	T\log|\Omega|
	-
	\frac{1}{|\Omega|}\int_\Omega U(x)\,dx
	-
	\frac{1}{|\Omega|}\int_\Omega u(x)\,dx.
	\]
	Using H\"older's inequality and Poincar\'e's inequality, we obtain
	\[
	\int_\Omega U(x)\,dx
	\leq
	|\Omega|^{1/2}\|U\|_{L^2(\Omega)}
	\leq
	C_\Omega \|\nabla U\|_{L^2(\Omega)}.
	\]
	Therefore it holds
	\[
	\mathcal E(U)
	\geq
	\frac12\|\nabla U\|_{L^2(\Omega)}^2
	-
	C_\Omega \|\nabla U\|_{L^2(\Omega)}
	-
	C,
	\]
	where \(C\) depends only on \(\Omega,T\), and \(u\). This shows that
	\(\mathcal E\) is bounded from below and coercive on the nonnegative cone of
	\(H_0^1(\Omega)\).
	
	\medskip
	
	\noindent
	\textit{Step 3: existence of a minimizer.}
	Let \((U_j)\) be a minimizing sequence. By Step 1, we may assume that
	\[
	U_j\geq 0
	\qquad\text{a.e. in }\Omega.
	\]
	By coercivity, \((U_j)\) is bounded in \(H_0^1(\Omega)\). Hence, up to a subsequence,
	there exists \(U\in H_0^1(\Omega)\) such that
	\[
	U_j\rightharpoonup U
	\qquad\text{weakly in }H_0^1(\Omega),
	\]
	and
	\[
	U_j\to U
	\qquad\text{a.e. in }\Omega.
	\]
	Since \(U_j\geq 0\), we also have \(U\geq 0\) a.e. in \(\Omega\).
	
	The Dirichlet term is weakly lower semicontinuous:
	\[
	\int_\Omega |\nabla U|^2\,dx
	\leq
	\liminf_{j\to\infty}
	\int_\Omega |\nabla U_j|^2\,dx.
	\]
	Moreover, it holds
	\[
	0\leq
	\rho_0(x)\exp\left(-\frac{U_j(x)}{T}\right)
	\leq
	\rho_0(x),
	\]
	and \(\rho_0\in L^1(\Omega)\). Since
	\[
	\rho_0(x)\exp\left(-\frac{U_j(x)}{T}\right)
	\to
	\rho_0(x)\exp\left(-\frac{U(x)}{T}\right)
	\quad\text{a.e. in }\Omega,
	\]
	Lebesgue's dominated convergence theorem gives
	\[
	\int_\Omega
	\rho_0(x)\exp\left(-\frac{U_j(x)}{T}\right)\,dx
	\to
	\int_\Omega
	\rho_0(x)\exp\left(-\frac{U(x)}{T}\right)\,dx.
	\]
	Therefore, we obtain 
	\[
	\mathcal E(U)
	\leq
	\liminf_{j\to\infty}\mathcal E(U_j),
	\]
	and \(U\) is a minimizer of \(\mathcal E\).
	
	\medskip
	
	\noindent
	\textit{Step 4: Euler-Lagrange equation.}
	Let \(\varphi\in H_0^1(\Omega)\cap L^\infty(\Omega)\). Since \(U\) is a minimizer,
	the function
	\[
	s\mapsto \mathcal E(U+s\varphi)
	\]
	has derivative equal to zero at \(s=0\). We compute
	\[
	\frac{d}{ds}\bigg|_{s=0}
	\frac12\int_\Omega |\nabla(U+s\varphi)|^2\,dx
	=
	\int_\Omega \nabla U\cdot\nabla\varphi\,dx.
	\]
	Moreover, we have
	\[
	\frac{d}{ds}\bigg|_{s=0}
	T\log\left(
	\int_\Omega
	\rho_0(x)\exp\left(-\frac{U(x)+s\varphi(x)}{T}\right)\,dx
	\right)
	=
	-
	\frac{
		\displaystyle\int_\Omega
		\rho_0(x)\exp\left(-\dfrac{U(x)}{T}\right)\varphi(x)\,dx
	}{
		\displaystyle\int_\Omega
		\rho_0(y)\exp\left(-\dfrac{U(y)}{T}\right)\,dy
	}.
	\]
	Therefore, we obtain the identity
	\[
	\int_\Omega \nabla U\cdot\nabla\varphi\,dx
	=
	\int_\Omega
	\frac{
		\rho_0(x)\exp\left(-\dfrac{U(x)}{T}\right)
	}{
		\displaystyle\int_\Omega
		\rho_0(y)\exp\left(-\dfrac{U(y)}{T}\right)\,dy
	}
	\varphi(x)\,dx.
	\]
	By density, this identity holds for every \(\varphi\in H_0^1(\Omega)\). Hence \(U\)
	is a weak solution of \((P)\). Since \(U\geq 0\), we have
	\[
	0\leq
	\frac{
		\rho_0(x)\exp\left(-\dfrac{U(x)}{T}\right)
	}{
		\displaystyle\int_\Omega
		\rho_0(y)\exp\left(-\dfrac{U(y)}{T}\right)\,dy
	}
	\leq
	\frac{\rho_0(x)}{A(U)}.
	\]
	Since \(\rho_0\in L^2(\Omega)\), the right-hand side of the equation belongs to
	\(L^2(\Omega)\). Elliptic regularity for the Dirichlet problem on \(C^{1,1}\) domains
	then gives
	\[
	U\in H^2(\Omega)\cap H_0^1(\Omega).
	\]
	
	\medskip
	
	\noindent
	\textit{Step 5: uniqueness by strict convexity.}
	We now prove that the minimizer is unique. Let \(U,V\in H_0^1(\Omega)\) and
	\(\theta\in(0,1)\). By H\"older's inequality, it follows that
	\[
	\begin{aligned}
		&\int_\Omega
		\rho_0(x)
		\exp\left(
		-\frac{\theta U(x)+(1-\theta)V(x)}{T}
		\right)\,dx
		\\
		&\qquad
		=
		\int_\Omega
		\left[
		\rho_0(x)\exp\left(-\frac{U(x)}{T}\right)
		\right]^\theta
		\left[
		\rho_0(x)\exp\left(-\frac{V(x)}{T}\right)
		\right]^{1-\theta}
		\,dx
		\\
		&\qquad
		\leq
		\left(
		\int_\Omega
		\rho_0(x)\exp\left(-\frac{U(x)}{T}\right)\,dx
		\right)^\theta
		\left(
		\int_\Omega
		\rho_0(x)\exp\left(-\frac{V(x)}{T}\right)\,dx
		\right)^{1-\theta}.
	\end{aligned}
	\]
	Thus the logarithmic term is convex. On the other hand, the Dirichlet term is
	strictly convex on \(H_0^1(\Omega)\). More precisely,
	\[
	\begin{aligned}
		\frac12
		\int_\Omega
		|\nabla(\theta U+(1-\theta)V)|^2\,dx
		&\leq
		\theta\frac12\int_\Omega |\nabla U|^2\,dx
		+
		(1-\theta)\frac12\int_\Omega |\nabla V|^2\,dx
		\\
		&\quad
		-
		\frac{\theta(1-\theta)}{2}
		\int_\Omega |\nabla(U-V)|^2\,dx.
	\end{aligned}
	\]
	Consequently, \(\mathcal E\) is strictly convex. Hence it has at most one minimizer.
	
	Since every minimizer solves \((P)\), and since the solution of \((P)\) is a critical
	point of \(\mathcal E\), the minimizer is the unique weak solution of the normalized
	Poisson-Boltzmann equation. This concludes the proof.
\end{proof}

\section{Optimal design of equilibrium densities}
\label{sec:optimal-design}

In this section, we formulate our general optimal design problem 
for equilibrium densities. The
aim is to determine an external potential \(u\), and hence an external electric field
\[
E_0=-\nabla u,
\]
such that the corresponding equilibrium density satisfies a prescribed design
criterion.

For a given external potential \(u\), we denote by \(U_u\) the unique solution of
problem \eqref{eq:Poisson-Boltzmann} corresponding to \(u\). The associated
normalized equilibrium density is given by
\[
\rho_u(x)
:=
\dfrac{
	\exp\left(-\dfrac{U_u(x)+u(x)}{T}\right)
}{
	\displaystyle\int_\Omega
	\exp\left(-\dfrac{U_u(y)+u(y)}{T}\right)\,dy
}.
\]
By construction, we have
\[
\rho_u\geq 0,
\qquad
\int_\Omega \rho_u(x)\,dx=1.
\]
Therefore, \(\rho_u\) is a probability density on \(\Omega\).

We introduce a general design datum, denoted by \(\eta\). Depending on the model,
\(\eta\) may represent either a valley potential \(V\) or a prescribed target density
\(\rho_d\). We then consider the corresponding objective functional
\[
G(\rho,\eta),
\]
which measures the design criterion imposed on the equilibrium density \(\rho\).

Let \(\alpha>0\). We define the cost functional
\[
J(\rho,u;\eta)
:=
\frac12 G(\rho,\eta)
+
\frac{\alpha}{2}\|u\|_{H^1(\Omega)}^2.
\]
The first term measures how well the equilibrium density satisfies the prescribed
design requirement, while the second term is a Tikhonov regularization term which
penalizes large or highly oscillatory external potentials.

This formulation can accommodate several relevant choices.
If \(\eta=V\), where \(V\in L^\infty(\Omega)\) is a valley potential, we may take
\[
G(\rho,V)
=
2\int_\Omega V(x)\rho(x)\,dx.
\]
Then
\[
J(\rho,u;V)
=
\int_\Omega V(x)\rho(x)\,dx
+
\frac{\alpha}{2}\|u\|_{H^1(\Omega)}^2,
\]
which corresponds to the valley-potential formulation.

On the other hand, if \(\eta=\rho_d\) is a prescribed target probability density,
one may take the classical tracking error functional
\[
G(\rho,\rho_d)
=
\|\rho-\rho_d\|_{L^2(\Omega)}^2,
\]
or the Kullback-Leibler divergence
\[
G(\rho,\rho_d)
=
\int_\Omega
\rho(x)\log\left(\frac{\rho(x)}{\rho_d(x)}\right)\,dx.
\]
In the latter case, we assume $\rho_d(x)>0$ for a.e. $x\in\Omega$, 
so that the logarithm is well defined whenever \(\rho=\rho_u\). This is consistent
with the fact that the equilibrium density generated by the normalized
Poisson-Boltzmann equation is strictly positive in \(\Omega\).

In the following, we keep \(G\) general and assume that it is G\^ateaux differentiable
with respect to its first argument. We denote 
\[
q_{\rho,\eta}(x)
:=
\frac12
\frac{\delta G}{\delta\rho}(\rho,\eta)(x).
\]
For the choices above, we have
\[
q_{\rho,V}(x)=V(x)
\]
in the valley-potential case,
\[
q_{\rho,\rho_d}(x)=\rho(x)-\rho_d(x)
\]
for the quadratic tracking term, and
\[
q_{\rho,\rho_d}(x)
=
\frac12
\left[
\log\left(\frac{\rho(x)}{\rho_d(x)}\right)+1
\right]
\]
for the Kullback-Leibler divergence.

Further, we introduce the admissible set
\[
\mathcal U_{\rm ad}
:=
\left\{
u\in H_0^1(\Omega):
m_1\leq u(x)\leq m_2
\ \text{for a.e. }x\in\Omega
\right\},
\]
where \(m_1,m_2\in\mathbb R\), \(m_1<m_2\), are prescribed constants. This choice
ensures, in particular, that
\[
e^{-u/T}\in L^\infty(\Omega)\subset L^2(\Omega),
\]
so that the well-posedness theory developed in the previous sections applies.

The optimal design problem is then the constrained minimization problem
\[
	\begin{aligned}
		&\min_{\rho,U,u}
		\left\{
		\frac12 G(\rho,\eta)
		+
		\frac{\alpha}{2}\|u\|_{H^1(\Omega)}^2
		\right\}
		\\
		&\text{subject to}
		\\
		&\rho(x)
		=
		\dfrac{
			\exp\left(-\dfrac{U(x)+u(x)}{T}\right)
		}{
			\displaystyle\int_\Omega
			\exp\left(-\dfrac{U(y)+u(y)}{T}\right)\,dy
		},
		\\[2ex]
		&-\Delta U=\rho
		\qquad\text{in }\Omega,
		\\
		&U=0
		\qquad\text{on }\partial\Omega,
		\\
		&u\in\mathcal U_{\rm ad}.
	\end{aligned}
\tag{OC}
\]

Equivalently, using the control-to-state map
\[
S:\mathcal U_{\rm ad}\to H^2(\Omega)\cap H_0^1(\Omega),
\qquad
S(u)=U_u,
\]
we may write the problem in reduced form. Namely, define
\[
\widehat J(u;\eta)
:=
J(\rho_u,u;\eta)=
\frac12 G(\rho_u,\eta)
+
\frac{\alpha}{2}\|u\|_{H^1(\Omega)}^2.
\]
Then the reduced optimal design problem is formulated as follows
\[
	\min_{u\in\mathcal U_{\rm ad}} \widehat J(u;\eta).
\tag{ROC}
\]

A minimizer \(u^\ast\) of \((ROC)\) is an optimal external potential. The
corresponding optimal equilibrium potential and density are
\[
U^\ast=S(u^\ast),
\quad \textrm{and}
\quad
\rho^\ast(x)
=
\dfrac{
	\exp\left(-\dfrac{U^\ast(x)+u^\ast(x)}{T}\right)
}{
	\displaystyle\int_\Omega
	\exp\left(-\dfrac{U^\ast(y)+u^\ast(y)}{T}\right)\,dy
}.
\]

\section{First-order optimality conditions}
\label{sec:first-order-optimality}

We now derive the formal first-order optimality conditions associated with the
general optimal design problem introduced above. We define
\[
\Phi(U,u)(x)
:=
\dfrac{
	\exp\left(-\dfrac{U(x)+u(x)}{T}\right)
}{
	\displaystyle\int_\Omega
	\exp\left(-\dfrac{U(y)+u(y)}{T}\right)\,dy
},\quad \textrm{and} \quad
\rho(x):=\Phi(U,u)(x).
\]
For a function \(h\), we introduce the notation 
$\langle h\rangle_\rho := \int_\Omega h(x)\rho(x)\,dx$. 

To compute the derivative of \(\Phi\) with respect to \(U\), set
\[
Z(U,u)
:=
\int_\Omega
\exp\left(-\frac{U(y)+u(y)}{T}\right)\,dy.
\]
Let \(\xi\) be an admissible variation of \(U\). Then
\[
D_U Z(U,u)[\xi]
=
-\frac1T
\int_\Omega
\exp\left(-\frac{U(y)+u(y)}{T}\right)\xi(y)\,dy.
\]
Since
\[
\exp\left(-\frac{U(y)+u(y)}{T}\right)
=
Z(U,u)\rho(y),
\]
we obtain
\[
D_U Z(U,u)[\xi]
=
-\frac{Z(U,u)}{T}
\int_\Omega \rho(y)\xi(y)\,dy.
\]
Using the quotient rule, we get that the Fréchet derivative of \(\Phi\) with respect to \(U\) is given by
\[
\begin{aligned}
	D_U\Phi(U,u)[\xi](x)
	&=
	-\frac1T\rho(x)\xi(x)
	+
	\frac1T\rho(x)
	\int_\Omega \rho(y)\xi(y)\,dy
	\\
	&=
	-\frac1T\rho(x)
	\left[
	\xi(x)-\langle \xi\rangle_\rho
	\right].
\end{aligned}
\]
The subtraction of the mean is a consequence of the normalization of
\(\Phi\). In particular, we obtain
\[
\int_\Omega D_U\Phi(U,u)[\xi](x)\,dx=0,
\]
as expected for the derivative of a probability density. Analogously, we have
\[
D_u\Phi(U,u)[h](x)
=
-\frac1T\rho(x)
\left[
h(x)-\langle h\rangle_\rho
\right].
\]

Next, we introduce the Lagrangian
\[
\mathcal L(U,u,p)
=
\frac12G(\rho,\eta)
+
\frac{\alpha}{2}\|u\|_{H^1(\Omega)}^2
+
\int_\Omega \nabla U\cdot\nabla p\,dx
-
\int_\Omega \rho\,p\,dx,
\]
where \(p\in H_0^1(\Omega)\) is the adjoint variable, namely the Lagrange multiplier
associated with the state equation

Taking variations with respect to \(U\), we obtain
\[
\int_\Omega \nabla \xi\cdot\nabla p\,dx
+
\int_\Omega
(q_{\rho,\eta}-p)D_U\Phi(U,u)[\xi]\,dx
=0 ,
\qquad
 \xi\in H_0^1(\Omega).
\]
Using the formula for \(D_U\Phi\), this becomes
\[
\int_\Omega \nabla \xi\cdot\nabla p\,dx
-
\frac1T
\int_\Omega
\rho(x)
\left[
(q_{\rho,\eta}(x)-p(x))
-
\langle q_{\rho,\eta}-p\rangle_\rho
\right]\xi(x)\,dx
=0.
\]
Therefore the adjoint equation is
\[
\begin{cases}
	-\Delta p
	=
	\dfrac1T
	\rho(x)
	\left[
	q_{\rho,\eta}(x)-p(x)
	-
	\displaystyle\int_\Omega \rho(y)(q_{\rho,\eta}(y)-p(y))\,dy
	\right],
	& x\in\Omega,\\[2ex]
	p=0,
	& x\in\partial\Omega.
\end{cases}
\tag{A}
\]

We now take variations with respect to the control \(u\). If the admissible set is
\[
\mathcal U_{\rm ad}
=
\left\{
u\in H_0^1(\Omega):
m_1\leq u\leq m_2
\ \text{a.e. in }\Omega
\right\},
\]
then the optimality condition is a variational inequality:
\[
\alpha (u,v-u)_{H^1(\Omega)}
+
\int_\Omega
(q_{\rho,\eta}-p)D_u\Phi(U,u)[v-u]\,dx
\geq 0 ,
\qquad
v\in\mathcal U_{\rm ad}.
\]
Using the expression of \(D_u\Phi\), we obtain
\[
\alpha (u,v-u)_{H^1(\Omega)}
-
\frac1T
\int_\Omega
\rho(x)
\left[
q_{\rho,\eta}(x)-p(x)
-
\langle q_{\rho,\eta}-p\rangle_\rho
\right]
(v(x)-u(x))\,dx
\geq 0
\]
for every \(v\in\mathcal U_{\rm ad}\).

We now collect these results that provide our first-order 
necessary optimality conditions. 
For convenience of notation, we define the centered quantity
\[
\mathcal M_\rho(x)
:=
q_{\rho,\eta}(x)-p(x)-\langle q_{\rho,\eta}-p\rangle_\rho.
\]
Thus our first-order optimality system is given by
\[
	\begin{aligned}
		-\Delta U &= \rho
		&&\text{in }\Omega,\\
		U&=0
		&&\text{on }\partial\Omega,\\[1ex]
		\rho(x)
		&=
		\dfrac{
			\exp\left(-\dfrac{U(x)+u(x)}{T}\right)
		}{
			\displaystyle\int_\Omega
			\exp\left(-\dfrac{U(y)+u(y)}{T}\right)\,dy
		}
		&& x\in\Omega,\\[2ex]
		-\Delta p
		&=
		\dfrac1T\,\rho\,\mathcal M_\rho
		&&\text{in }\Omega,\\
		p&=0
		&&\text{on }\partial\Omega.
	\end{aligned}
\tag{OS}
\]
The optimality condition for the control reads
\[
	\begin{aligned}
		&\alpha (u,v-u)_{H^1(\Omega)}
		-
		\frac1T
		\int_\Omega
		\rho(x)\mathcal M_\rho(x)(v(x)-u(x))\,dx
		\geq 0,
	\end{aligned}
\tag{VI}
\label{VI}
\]
for all $v\in\mathcal U_{\rm ad}$.

We remark that, if no pointwise constraints are imposed, or if the 
control constraints are not active, then the variational inequality becomes the equation
\[
\alpha (u,h)_{H^1(\Omega)}=
\frac1T
\int_\Omega
\rho(x)\mathcal M_\rho(x)h(x)\,dx ,
\qquad
 h\in H_0^1(\Omega).
\tag{OE}
\]
If the \(H^1\)-inner product is defined by
\[
(u,h)_{H^1(\Omega)}
=
\int_\Omega \nabla u\cdot\nabla h\,dx
+
\int_\Omega uh\,dx,
\]
then \((OE)\) can be written in weak form as follows
\[
\alpha
\int_\Omega \nabla u\cdot\nabla h\,dx
+
\alpha
\int_\Omega uh\,dx
=
\frac1T
\int_\Omega
\rho\,\mathcal M_\rho h\,dx ,
\qquad
h\in H_0^1(\Omega).
\]
Thus, formally, at optimality the optimal control satisfies the following equation
\[
	\alpha(-\Delta u+u)
	=
	\frac1T\rho\,\mathcal M_\rho
	\qquad\text{in }\Omega,
\tag{CE}
\]
with homogeneous Dirichlet boundary condition.

\section{Numerical results}

We illustrate the performance of the proposed optimization framework by
considering the three design criteria introduced in Section~6. The
computational domain is the unit disk
\[
\Omega=\{x\in\mathbb R^2:|x|<1\},
\]
and the temperature and Tikhonov parameter are fixed as
\[
T=1,
\qquad
\alpha=10^{-4}.
\]
The state and the control are initialized at zero. The computations are
performed in polar coordinates on a grid with $N_r=100$ radial nodes and
$N_\varphi=120$ angular nodes; the radial grid is refined in order to better
resolve the transition layers of the design data. The nonlinear state equation
is solved by fixed-point iteration, while the reduced optimization problem is
solved by the nonlinear conjugate-gradient method described in Section~8. The
optimization is stopped when the $H^1(\Omega)$-norm of the reduced gradient is
smaller than $10^{-6}$.

The design region is the annulus determined by
\[
r_{\mathrm{in}}=\frac{2}{5},
\qquad
r_{\mathrm{out}}=\frac{1}{2}.
\]
To obtain smooth design data, we introduce
\[
\chi_\varepsilon(r)
=
\frac12\left[
\tanh\left(\frac{r-r_{\mathrm{in}}}{\varepsilon}\right)
-
\tanh\left(\frac{r-r_{\mathrm{out}}}{\varepsilon}\right)
\right],
\qquad
\varepsilon=10^{-2}.
\]
For the valley-potential functional, we take
\[
V(r)=-\chi_\varepsilon(r),
\]
whereas the common target density used in the quadratic and
Kullback--Leibler cases is
\[
\rho_d(r)
=
\frac{\chi_\varepsilon(r)+\rho_{\mathrm b}}
{\displaystyle\int_\Omega
	\bigl(\chi_\varepsilon(|x|)+\rho_{\mathrm b}\bigr)\,dx},
\qquad
\rho_{\mathrm b}=10^{-6}.
\]
The small positive background $\rho_{\mathrm b}$ guarantees that the target
is strictly positive, as required by the Kullback--Leibler functional.

\begin{figure}[htbp]
	\centering
	\includegraphics[width=\textwidth]{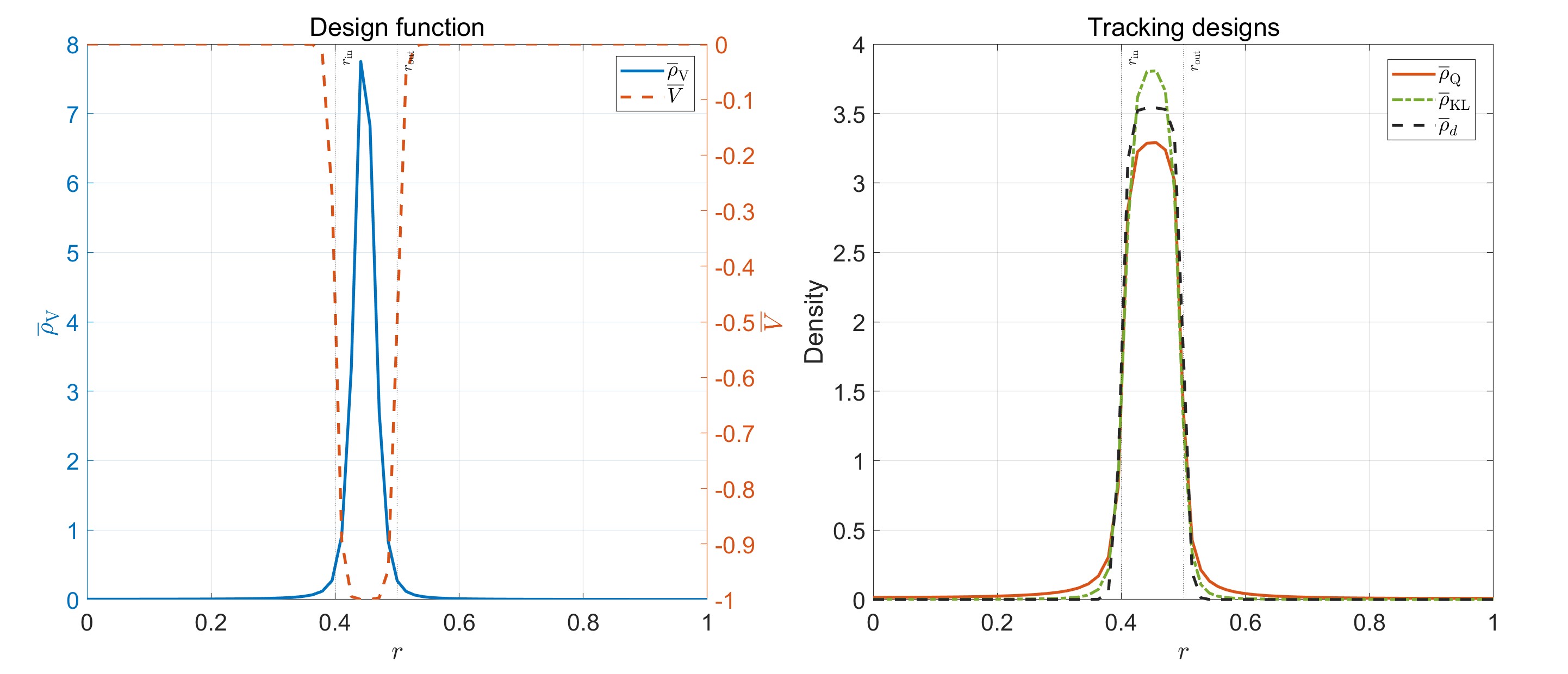}
	\caption{Angularly averaged design data and  optimal densities. Left: the
		optimal density $\overline\rho_{\mathrm V}$ obtained with the valley-potential
		functional and the corresponding design function $\overline V$. Right: the
		optimal densities $\overline\rho_{\mathrm Q}$ and
		$\overline\rho_{\mathrm{KL}}$ obtained with the quadratic and
		Kullback--Leibler criteria, respectively, together with their common target
		$\overline\rho_d$. }
	\label{fig1}
\end{figure}

Figure~\ref{fig1} highlights the different meaning of the three design
criteria. The valley-potential functional does not prescribe a complete target
profile; it only rewards the placement of mass in the region where $V$ is
smallest. Accordingly, the resulting density is sharply concentrated near the
middle of the prescribed annulus and reaches a considerably larger maximum
than the two tracking solutions. The quadratic and Kullback--Leibler criteria,
on the other hand, reproduce both the position and the width of the desired
annular profile. Their densities are very close to one another and to $\rho_d$,
with only small discrepancies near the transition layers and on the plateau.
The Kullback--Leibler solution is slightly sharper, whereas the quadratic
tracking solution is somewhat more diffuse near the edges of the annulus.

\begin{figure}[htbp]
	\centering
	\includegraphics[width=\textwidth]{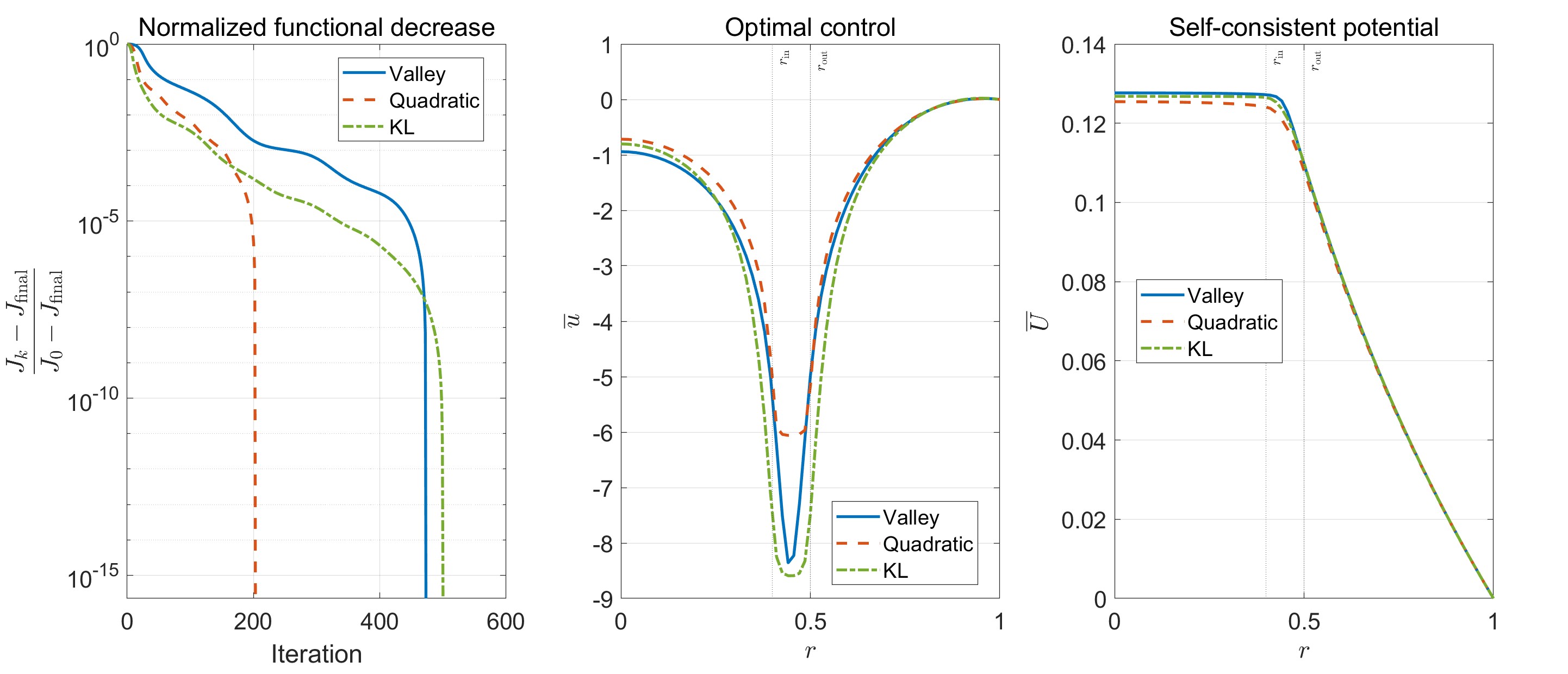}
	\caption{Comparison of the three design problems. Left: normalized decrease
		of the cost functional,
		$(J_k-J_{\mathrm{final}})/(J_0-J_{\mathrm{final}})$. Center: angularly
		averaged optimal controls. Right: angularly averaged self-consistent
		potentials.}
	\label{fig2}
\end{figure}

The left panel of Figure~\ref{fig2} shows a monotone decrease of the cost in
all three cases, consistently with the Armijo line search. The quadratic
tracking problem reaches its final numerical value in fewer iterations,
whereas the valley-potential and Kullback--Leibler problems exhibit a more
gradual decrease. In all the experiments, the trial step length $s=2$ is
accepted without backtracking. The central panel shows that the main
differences among the three optimal designs are carried by the external
control. The quadratic criterion requires a noticeably shallower potential
well, while the valley and Kullback--Leibler criteria produce stronger
negative controls around the design annulus. By contrast, the three
self-consistent potentials displayed in the right panel are nearly
indistinguishable. This reflects the smoothing action of the Poisson equation
and the fact that all three normalized densities have unit mass and are
localized in the same radial region.

\begin{figure}[htbp]
	\centering
	\includegraphics[width=\textwidth]{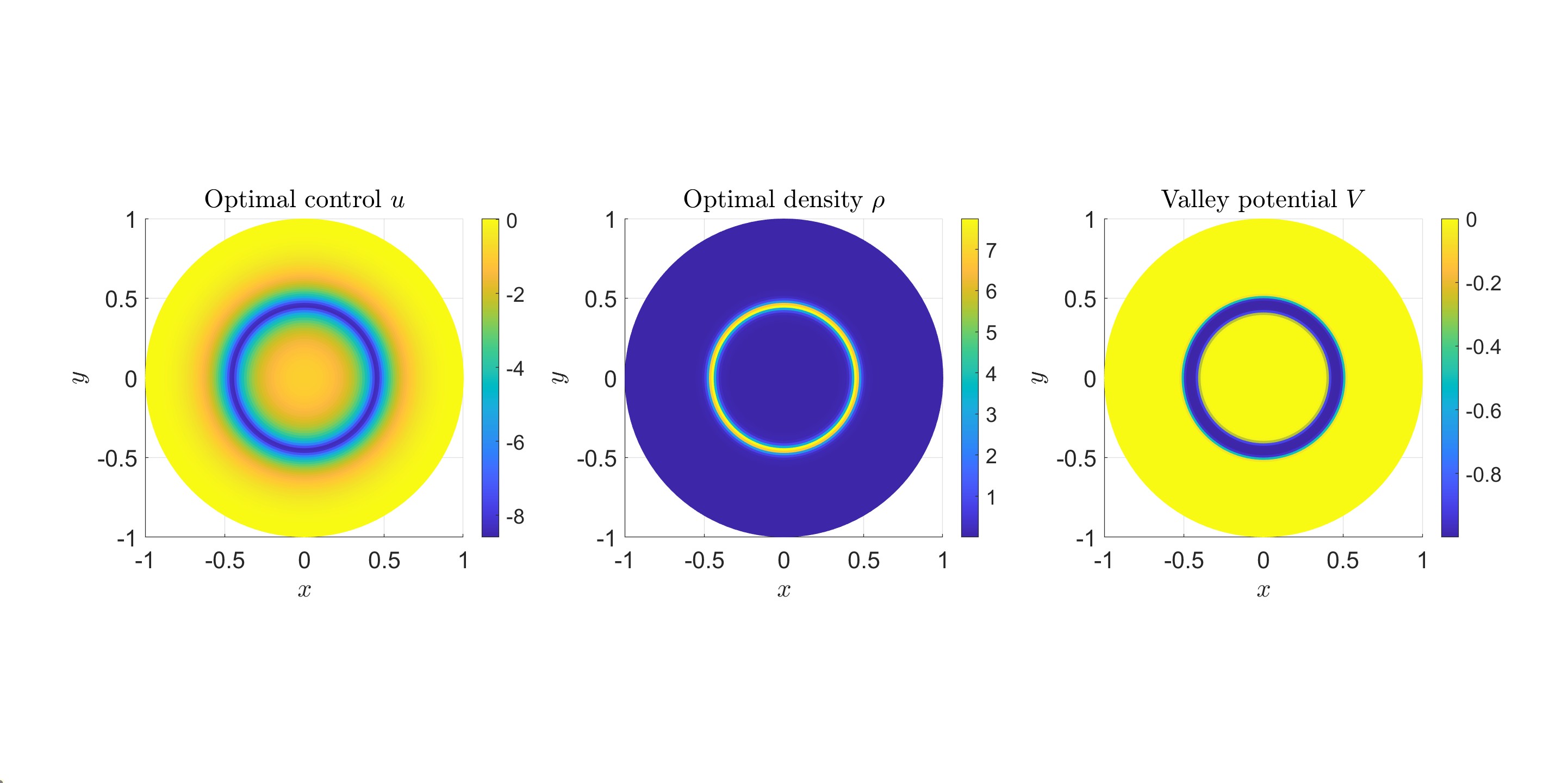}
	\caption{Valley-potential design. From left to right: optimal external potential $u$, optimal equilibrium
		density $\rho$,  and valley potential $V$.}
	\label{fig3}
\end{figure}

The two-dimensional fields for the valley-potential case are shown in
Figure~\ref{fig3}. The optimized control creates a negative well around the
prescribed annulus, producing a density concentrated along a thinner ring.
This occurs because the term \(\int_{\Omega}V\rho\,dx\) favours mass
concentration where \(V\) is minimal, without enforcing the reproduction of
a prescribed density profile.

\begin{figure}[htbp]
	\centering
	\includegraphics[width=\textwidth]{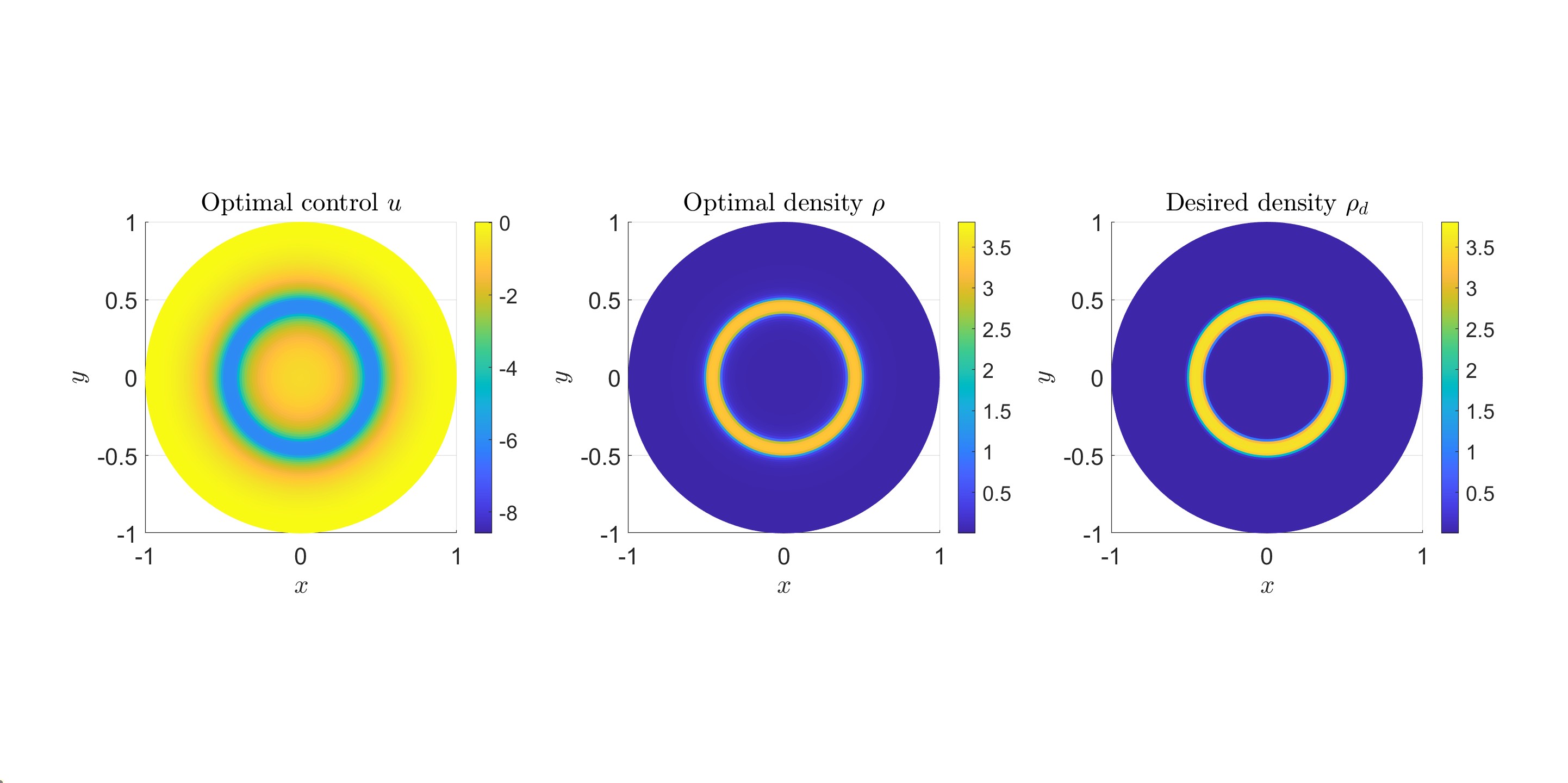}
	\caption{Quadratic tracking design. From left to right: optimal external potential $u$, optimal equilibrium
		density $\rho$,  and desired density $\rho_d$.}
	\label{fig4}
\end{figure}

Figure~\ref{fig4} reports the result for the quadratic tracking functional.
The optimal density correctly reproduces the annular support, radius, and
width of $\rho_d$. Its maximum is slightly lower than that of the target, and
a small amount of mass is distributed near the transition regions. This is
the expected compromise between the $L^2$ tracking term and the $H^1$
regularization of the control. In comparison with the other two criteria, the
quadratic functional achieves the prescribed configuration with the least
intense external potential.

\begin{figure}[htbp]
	\centering
	\includegraphics[width=\textwidth]{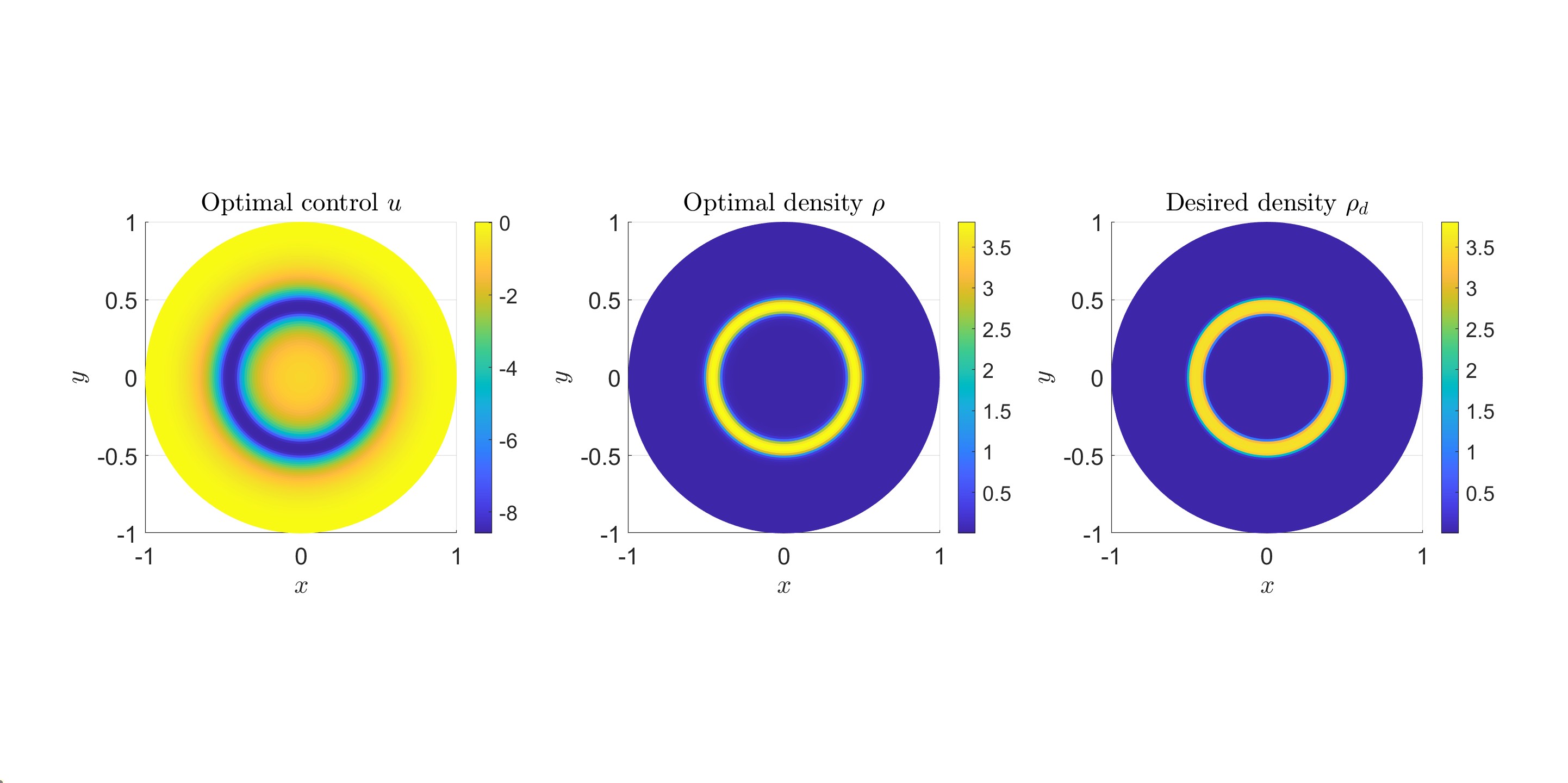}
	\caption{Kullback--Leibler design. From left to right:, optimal external potential $u$, optimal equilibrium
		density $\rho$, and desired density $\rho_d$.}
	\label{fig5}
\end{figure}

The Kullback--Leibler result is displayed in Figure~\ref{fig5}. The optimal
density provides a very accurate reconstruction of the desired annular
configuration and suppresses density leakage in the region where $\rho_d$ is
small. This behaviour is consistent with the relative character of the
Kullback--Leibler divergence, which strongly penalizes discrepancies in
low-target regions. The improved confinement is obtained at the price of a
deeper external potential well than in the quadratic case.

Overall, the numerical experiments confirm that the same state, adjoint, and
gradient framework can accommodate substantially different notions of optimal
design. The valley-potential criterion is particularly effective when only a
preferred confinement region is prescribed, while the quadratic and
Kullback--Leibler criteria allow direct approximation of a desired equilibrium
density. Among the two tracking formulations, the quadratic functional uses a
milder control, whereas the Kullback--Leibler divergence yields a sharper
relative match to the target. The close agreement of the corresponding
self-consistent potentials also indicates that appreciably different external
fields can produce equilibrium configurations with similar macroscopic
self-consistent electrostatic response.

\end{document}